\documentclass[final]{siamltex}
\usepackage{amsmath,amsfonts,amssymb}
\usepackage{latexsym}
\usepackage{graphicx,overpic}
\usepackage{subfig}
\usepackage{pgfplots}
\pgfplotsset{compat=newest}
\usepackage{color}
\usepackage{booktabs}
\usepackage{float}
\usepackage{showlabels}

\newcommand{\err}{\mathbb{E}}
\newcommand{\R}{\mathbb R}
\newcommand{\bA}{\mathbf A}

\newcommand{\bH}{\mathbf H}
\newcommand{\bI}{\mathbf I}

\newcommand{\bP}{\mathbf P}

\newcommand{\bV}{\mathbf V}

\newcommand{\blf}{\mathbf f}
\newcommand{\bn}{\mathbf n}
\newcommand{\be}{\mathbf e}
\newcommand{\bp}{\mathbf p}

\newcommand{\bu}{\mathbf u}
\newcommand{\bU}{\mathbf U}
\newcommand{\bv}{\mathbf v}
\newcommand{\bw}{\mathbf w}

\newcommand{\bx}{\mathbf x}

\newcommand{\T}{\mathcal T}

\newcommand{\Div}{\mathop{\rm div}}
\newcommand{\divG}{{\mathop{\,\rm div}}_{\Gamma}}
\newcommand{\gradG}{\nabla_{\Gamma}}
\newcommand{\nablaG}{\nabla_{\Gamma}}

\newcommand{\OGamma}{\Omega^\Gamma_h}

\renewcommand{\div}{\textrm{div}\ \!}

\newcommand{\tr}{{\rm tr}}
\DeclareGraphicsExtensions{.pdf,.eps,.ps,.eps.gz,.ps.gz,.eps.Y}

\newcommand{\bxi}{\mbox{\boldmath$\xi$\unboldmath}}
\newcommand{\bpsi}{\mbox{\boldmath$\psi$\unboldmath}}

\newtheorem{assumption}{Assumption}[section]

\newtheorem{remark}{Remark}[section]

\begin{document}
\title{A penalty finite element method for a fluid system posed on embedded surface}
\author{
Maxim A. Olshanskii\thanks{Department of Mathematics, University of Houston, Houston, Texas 77204 (molshan@math.uh.edu); Partially supported by NSF through the Division of Mathematical Sciences grants 1522252 and 1717516.}
\and
Vladimir Yushutin\thanks{Department of Mathematics, University of Houston, Houston, Texas 77204 (yushutin@math.uh.edu)}
}
\maketitle
%
\begin{abstract}
The paper introduces a finite element method for the incompressible Navier--Stokes equations posed on a closed surface $\Gamma\subset\R^3$. The method needs a shape regular tetrahedra mesh in $\R^3$ to discretize equations on the surface, which can cut through this mesh in a fairly  arbitrary way. Stability and error analysis of the fully discrete (in space and in time) scheme is given. The tangentiality condition for the velocity field on $\Gamma$ is enforced weakly by a penalty term. The paper studies both theoretically and numerically the dependence of the error on the penalty parameter. Several numerical examples demonstrate convergence and conservation properties of the finite element method.
\end{abstract}

\section{Introduction}

Fluid equations posed on manifolds naturally arise in mathematical models of lipid membranes, foams, emulsions and other thin material layers that exhibit surface fluidity and viscosity; see, e.g., \cite{baumgart2003imaging,dickinson1999adsorbed,fan2010hydrodynamic,scriven1960dynamics,slattery2007interfacial}.
Recently there has been a growing interest to  numerical simulation of fluid systems posed on surfaces~\cite{barrett2014stable,fries2017higher,nitschke2017discrete,nitschke2012finite,olshanskii2018finite,reuskenstream,reuther2015interplay,reuther2018solving,rodrigues2015semi}.
Due to its geometrical flexibility, finite element method is the most popular numerical approach for surface Darcy, Stokes, Navier--Stokes and coupled bulk--surface fluid problems. For example, papers \cite{fries2017higher,olshanskii2018finite,reuther2018solving} apply surface finite element methods ($P_1$-$P_1$ in \cite{reuther2018solving,olshanskii2018finite} and  Taylor--Hood elements in \cite{fries2017higher}) to discretize the incompressible surface Navier-Stokes equations in primitive variables on stationary manifolds. The authors of   \cite{nitschke2012finite,reuther2015interplay,reuskenstream} rewrite the governing equations  in vorticity--stream function variables, which are scalar quantities for  2D surfaces,  and further apply $P_1$ finite element methods to the resulting system. In \cite{barrett2014stable} a steady coupled bulk--surface Navier-Stokes system was also  treated by a finite element method. The present paper contributes to this emerging research field  with  stability and error analysis of a geometrically unfitted finite element method introduced here for the  Navier--Stokes equations of incompressible viscous surface fluid.

Discretization of fluid systems on manifolds brings up several difficulties in addition to those well-studied for finite element methods applied to equations posed in Euclidian domains. First, one has to approximation of covariant derivatives. The present paper exploits embedding of the two-dimensional  surface in $\R^3$ and makes use of tangential differential calculus; see, e.g., \cite{Jankuhn1,Gigaetal,scriven1960dynamics} for the derivation of surface fluid equations in terms of exterior differential operators in Cartesian coordinates. This allows us to avoid the use of intrinsic variables on a surface and makes implementation of the numerical method relatively straightforward in a standard finite element software. Next, in certain computational approaches -- for example, based on vorticity--stream function variables -- to formulate a finite element method, one has to recover surface curvatures, which is known to be a delicate numerical procedure, unless they are explicitly available through surface parametrization. The present method does not need this information and is capable of handling systems posed on implicitly defined surfaces.
Another difficulty stems from the need to recover a \emph{tangential} velocity field on a surface $\Gamma$.    It is not straightforward to build a finite element method which is conformal with respect to this tangentiality condition. In fact, the subspace of finite element velocity functions $\bu_h$ satisfying $\bu_h\cdot\bn=0$ ($\bn$ is the normal vector field on $\Gamma$) may contain only zero elements (geometrical locking effect).
Two natural ways to enforce the condition in the numerical setting are either the use of Lagrange multipliers \cite{gross2017trace} or adding a penalty term to the weak variational formulation.
Following \cite{hansbo2016stabilized,hansbo2016analysis,Jankuhn1,olshanskii2018finite,reuther2018solving}, we shall enforce the tangential constraint weakly with the help of a penalty term. Finally, one has to deal with geometric errors originating from approximation of $\Gamma$ by a ``discrete'' (e.g. polygonal) surface $\Gamma_h$ or, more general, from inexact integration of finite element bilinear forms over $\Gamma$. 
The effect of this geometric consistency on a  finite element  error for  surface vector Laplacian equation was studied in \cite{hansbo2016analysis}. For finite element exterior calculus approximation of the Hodge-Laplacian operator, the geometric consistency estimates were derived in \cite{holst2012geometric}. We do not address this issue here, assuming exact integration over $\Gamma$.

The present paper builds on the earlier work on the unfitted trace finite element method (TraceFEM) for PDEs posed on manifolds embedded in $\R^d$, $d=2,3$; in particular we exploit certain ideas found in \cite{burman2016cutb,grande2017higher,gross2017trace,olshanskii2018finite,ORG09}. The method uses shape regular surface-independent background triangulation and allows a surface or a curve to cut through this triangulation in an arbitrary way.
The choice of the geometrically  unfitted discretization  is motivated by the ultimate goal of numerical simulation of fluid flows on \emph{evolving} surfaces $\Gamma(t)$~\cite{Jankuhn1,Gigaetal,yavari2016nonlinear}. Unfitted discretizations, such as TraceFEM, allow to avoid mesh reconstruction for the time-dependent geometry and
to treat implicitly defined surfaces. As illustrated, for example, in \cite{lehrenfeld2018stabilized},
TraceFEM works very well for scalar PDEs posed on evolving surfaces, including cases where $\Gamma(t)$ undergoes topological changes, and it can be naturally combined with the level set method for implicit surface representation.

The paper presents a complete error analysis of the TraceFEM for time-dependent incompressible Navier-Stokes equations on a steady surface. Previous numerical analyses of fluid and related systems on manifolds include
error analysis of fitted and unfitted finite element methods for surface vector-Laplacian problems in \cite{hansbo2016analysis} and \cite{gross2017trace}, respectively,  as well as the error analysis of $P_1-P_1$ TraceFEM for the steady Stokes problem in~\cite{olshanskii2018finite}. Thus, the novelty here is the analysis of a time-dependent fluid system and the inclusion of inertia terms. Furthermore, we allow the surface to have non-trivial vector fields of infinitesimal rigid transformations. The corresponding velocity vector fields belong to the kernel of the viscous term and so the PDE system is not dissipative on the whole  space of divergence free tangential velocities, but only on a subspace. The finite element method preserves the corresponding property  only approximately, and handling it  requires some less standard considerations.

The remainder of the paper is organized in four sections. In section~\ref{s_cont} we recall some elementary notions of tangential calculus, introduce the surface incompressible Navier-Stokes equations and their weak formulation. We further derive energy balance and basic a priori estimates, which should be helpful in understanding expected properties of  the discrete problem. Section~\ref{s_TraceFEM} introduces the fully discrete finite element  formulation and discusses necessary implementation details. Stability and error analysis is the topic of section~\ref{s_error}. Here we prove an error estimate of optimal order in the energy norm. We also  track carefully the dependance of the error estimate on the penalty parameter. This reveals the optimal scaling of this parameter with respect to discretization parameters. Finally, section~\ref{s_num} presents results of a few numerical experiments, which illustrate the theory.

\section{Continuous problem}\label{s_cont}
Assume that $\Gamma$ is a closed sufficiently smooth surface in $\mathbb{R}^3$. The outward pointing unit normal on $\Gamma$ is denoted by $\bn$, and the orthogonal projection on the tangential plane is given by $\bP=\bP(\bx):= \bI - \bn(\bx)\bn(\bx)^T$, $\bx \in \Gamma$. In a neighborhood $\mathcal{O}(\Gamma)$  of $\Gamma$ the closest point projection $\bp:\,\mathcal{O}(\Gamma)\to \Gamma$ is well defined.  For a scalar function $p:\, \Gamma \to \mathbb{R}$ or a vector function $\bu:\, \Gamma \to \mathbb{R}^3$  we define $p^e=p\circ \bp\,:\,\mathcal{O}(\Gamma)\to\mathbb{R}$, $\bu^e=\bu\circ \bp\,:\,\mathcal{O}(\Gamma)\to\mathbb{R}^3$,   extensions of $p$ and $\bu$ from $\Gamma$ to its neighborhood $\mathcal{O}(\Gamma)$ along the normal directions.
The surface gradient and covariant derivatives on $\Gamma$ are then defined as $\nablaG p=\bP\nabla p^e$ and  $\nabla_\Gamma \bu:= \bP \nabla \bu^e \bP$. The definitions  of surface gradient and covariant derivatives are  independent of a particular smooth extension of $p$ and $\bu$ off $\Gamma$.
On $\Gamma$ we consider the surface rate-of-strain tensor \cite{GurtinMurdoch75} given by
\begin{equation} \label{strain}
 E_s(\bu):= \frac12 \bP (\nabla \bu +\nabla \bu^T)\bP = \frac12(\nabla_\Gamma \bu + \nabla_\Gamma \bu^T).
 \end{equation}
We also define the surface divergence operators for a vector $\bu: \Gamma \to \R^3$ and
a tensor $\bA: \Gamma \to \mathbb{R}^{3\times 3}$:
\[
 \divG \bu  := \tr (\gradG \bu), \qquad
 \divG \bA  := \left( \divG (\be_1^T \bA),\,
               \divG (\be_2^T \bA),\,
               \divG (\be_3^T \bA)\right)^T,
               \]
with $\be_i$ the $i$th basis vector in $\R^3$.

For a given  force vector $\mathbf{f} \in L^2(\Gamma)^3$, with $\mathbf{f}\cdot\bn=0$, we consider the following surface Navier--Stokes problem:
Find a vector field $\bu:\, \Gamma \to \R^3$, with $\bu\cdot\bn =0$, and $p:\,\Gamma \to \R$ such that
\begin{align} 
  \rho\left(\frac{\partial\bu}{\partial t}+(\nabla_\Gamma \bu)\bu\right) - \nu\bP \divG (E_s(\bu))+\nabla_\Gamma p &=  \blf \quad \text{on}~\Gamma,  \label{strongform-1} \\
  \divG \bu & =0 \quad \text{on}~\Gamma. \label{strongform-2}
\end{align}
Here $\bu$ is the tangential fluid velocity, $p$  the surface fluid pressure, $\rho$ and $\nu$ are density and viscosity coefficients  and $\dot{\bu}=\frac{\partial\bu}{\partial t}+(\nabla_\Gamma \bu)\bu$ is the full time derivative, i.e. derivative along material trajectories of surface particles. We further assume $\nu$, $p$ and $\blf$ re-scaled so that $\rho=1$.

\begin{remark} \label{remLaplacian} \rm
The operator $ \bP \divG E_s(\cdot)$ in equation~\eqref{strongform-1} models surface diffusion, which
is a key component in modeling Newtonian surface fluids and fluidic membranes \cite{GurtinMurdoch75,scriven1960dynamics}.
In the literature, there are different formulations of the surface Navier--Stokes equations,
some of which are formally obtained by substituting Cartesian differential operators by their geometric counterparts.  These formulations may involve different surface Laplace type operators, e.g., Bochner or Hodge--de Rham Laplacians. We refer to \cite{Jankuhn1} for a brief overview of different formulations of  the surface Navier--Stokes equations.
\end{remark}
\smallskip


\subsection{Weak formulation}\label{s_weak} We assume $\Gamma$ is at least $C^2$ smooth and compact.
Further $(\cdot,\cdot)$ and $\|\cdot\|$ denote $L^2(\Gamma)$ inner product and norm.
In what follows, we  need both general and tangential vector fields on $\Gamma$. Hence, we consider the space $\bV:=H^1(\Gamma)^3$ with norm
 \begin{equation} \label{H1norm}
  \|\bu\|_{1}^2:=\|\bu\|^2 + \|\nabla_
 \Gamma\bu\|^2,
 \end{equation}
and its  subspaces
\begin{equation}   \label{defVT}
 \bV_T:= \{\, \bu \in \bV~|~ \bu\cdot \bn =0\,\},\quad E:= \{\, \bu \in \bV_T~|~ E_s(\bu)=\mathbf{0}\,\}.
\end{equation}
For $\bu \in \bV$ we define the orthogonal decomposition into tangential and normal parts:
\begin{equation}\label{u_T_N}
\bu = \bu_T + u_N\bn,\quad \bu_T\in \bV_T.
\end{equation}
We use the notation from \eqref{u_T_N} further in the text.
Subspace $E$ from \eqref{defVT} spans all  infinitesimal rigid transformations -- also known as Killing vector fields -- that may exist for $\Gamma$. $E$ is a closed subspace of $\bV_T$ and $\mbox{dim}(E)\le 3$ for a two-dimensional manifold; see, e.g.~\cite{sakai1996riemannian}.
We define the Hilbert space $\bV_T^0$ as an orthogonal complement of  $E$ in $\bV_T$.
We also define $L_0^2(\Gamma):=\{\, p \in L^2(\Gamma)~|~ \int_\Gamma p\,dx=0\,\}$.

Consider the bilinear forms (with $A:B={\rm tr}\big(AB^T\big)$ for  $A,B\in\mathbb{R}^{3\times3}$)
\begin{align}
c(\bw,\bu,\bpsi)&=\int_{\Gamma} \left((\nabla_\Gamma\bu)
\bw\right)\cdot \bpsi\,\mathrm{d}\bx, \quad \bw,\bu,\bpsi \in \bV\label{defblfc}  \\
a(\bu,\bv)& = \int_\Gamma E_s(\bu):E_s(\bv) \, ds, \quad \bu,\bv \in \bV, \label{defblfa} \\
b(\bu,p) &= - \int_\Gamma p\,\divG \bu \, ds,  \quad \bu \in \bV, ~p \in L^2(\Gamma). \label{defblfb}
\end{align}

A weak formulation of \eqref{strongform-1}--\eqref{strongform-2}  reads: Find $\bu\in L^2(0,T;\bV_T(\Gamma))$ with $\bu_t\in L^2(0,T;\bV_T'(\Gamma))$ and $p\in L^2(\Gamma\times(0,T))$ satisfying $\bu|_{t=0}=\bu_0$ on $\Gamma$ and
\begin{equation}\label{weak}
\langle\bu_t,\bpsi\rangle_{\bV_T\times\bV_T'} +c(\bu,\bv,\bpsi)+a(\bu,\bpsi)  -b(p,\bpsi)+b(q,\bu)=(\blf,\bpsi)
\end{equation}
for all $\bpsi\in \bV_T,$ $q\in L^2(\Gamma)$ and for a.e. $t\in[0,T]$.

For divergence free tangential vector fields, we find using integration by parts that the $c$-form is skew-symmetric in its second and third arguments:
\begin{equation} \label{c_skew}
c(\bu,\bv,\bpsi)=-c(\bu,\bpsi,\bv),\quad\forall~\bu,\bv,\bpsi\in\bV_T,~ \mbox{div}_\Gamma\bu=0.
\end{equation}

Testing \eqref{weak} with $\bpsi=\bu$, $q=-p$, and using \eqref{c_skew} we obtain the energy balance equality,
\begin{equation}\label{balance_eq}
\frac12\frac{d}{dt}\|\bu\|^2+\nu\|E_s(\bu)\|^2=(\blf,\bu),\quad a.e.~~ t\in[0,T].
\end{equation}
Also for any $\bv\in E$ we have
\begin{equation}\label{c_killing}
c(\bu,\bu,\bv)=-c(\bu,\bv,\bu)=-\int_{\Gamma} \bu^T(\nabla_\Gamma \bv)\bu\,ds=-\int_{\Gamma} \bu^T E_s(\bv)\bu\,ds=0.
\end{equation}
Consider decomposition $\bu=\bu^0+\bu^e$,
$\bu^0\in\bV_T^0$,  $\bu^e\in E$ for all $t\in[0,T]$ and test \eqref{weak} with $\bpsi=\bu^e$, $q=-p$. Thanks to \eqref{c_skew} and \eqref{c_killing} we have
\[
\left.
\begin{split}
   c(\bu,\bu,\bu) & =0 \\
   c(\bu,\bu,\bu^e)&=0
\end{split}
\right\}\quad\Rightarrow\quad
c(\bu,\bu,\bu^0)=0.
\]
This and $E_s(\bu^e)=0$ yield  the energy  balance both for $\bu^e$ and $\bu^0$ parts of the solution,
\begin{equation}\label{balance_separate}
\frac12\frac{d}{dt}\|\bu^e\|^2=(\blf,\bu^e)\quad \text{and}\quad \frac12\frac{d}{dt}\|\bu^0\|^2+\nu\|E_s(\bu^0)\|^2=(\blf,\bu^0).
\end{equation}
We see that system is dissipative on $\bV_T^0$, but not on the whole space $\bV_T$, if $\mbox{dim}(E)>0$. 
The estimate
\[
\|\bu\|_{L^\infty(0,T; L^2(\Gamma))}\le \|\bu_0\|+2\,\|\blf\|_{L^1(0,T; L^2(\Gamma))}
\]
follows immediately from \eqref{balance_eq}.  To show a bound for the $L^2(0,T; \bV)$ norm of $\bu$,
we need the surface Korn inequality below.
 There exist $c_K >0$ such that
 \begin{equation} \label{korn}
\|E_s(\bv)\| \geq c_K \|\bv\|_{1} \quad \text{for all}~~\bv \in \bV_T^0;
 \end{equation}
see \cite{Jankuhn1}. Since $E$ is finite dimensional (and so all norms on $E$ are equivalent), inequality \eqref{korn} implies
 \begin{equation} \label{korn1}
\|\bv\|^2+ \|E_s(\bv)\|^2 \geq C_K \|\bv\|_{1}^2 \quad \text{for all}~~\bv \in \bV_T.
 \end{equation}
Now with the help of the Cauchy--Schwarz inequality we conclude from \eqref{balance_eq} that
\begin{multline*}
\|\bu(t)\|^2\le  \|\bu_0\|^2+2\int_0^t\|\blf(s)\|_{-1}\|\bu(s)\|_{1}\,ds~~\Rightarrow~\\
~\Rightarrow~~\int_0^t\|\bu(s)\|^2\,ds\le t\|\bu_0\|^2+2t\int_0^t\|\blf(s)\|_{-1}\|\bu(s)\|_{1}\,ds,
\end{multline*}
where $\|\cdot\|_{-1}$ is the dual norm for $\bV$-norm.
From \eqref{balance_eq} we also have
\[\nu\int_0^t\|E_s(\bu(s))\|^2\,ds\le \frac12\|\bu_0\|^2+\int_0^t\|\blf(s)\|_{-1}\|\bu(s)\|_{1}\,ds.\]
Therefore, thanks to \eqref{korn1}, we get
\begin{equation*}
\begin{split}
C_K^{-1}&\nu\int_0^t\|\bu(s)\|^2_1\,ds \le \nu\int_0^t(\|\bu(s)\|^2+\|E_s(\bu(s))\|^2)\,ds \\ &\le (\frac12+\nu t)\|\bu_0\|^2+(1+2t\nu)\int_0^t\|\blf(s)\|_{-1}\|\bu(s)\|_{1}\,ds\\
&\le(\frac12+\nu t)\|\bu_0\|^2+(1+2t\nu)^2C_K\nu^{-1}\int_0^t\|\blf(s)\|^2_{-1}\,ds +\frac{1}4 C_K^{-1}\nu\int_0^t\|\bu(s)\|_{1}^2\,ds.
\end{split}
\end{equation*}
After  cancellation, this implies
\begin{equation}\label{energy1}
\|\bu\|_{L^2(0,T; \bV)}\le c\,(\|\bu_0\|+ \|\blf\|_{L^2(0,T; \bV')}).
\end{equation}
A discrete counterpart of \eqref{energy1} will be important for the error analysis further in the paper.

\subsection{Some further useful results}
We shall also need the space
\[
\bV_\ast  :=\{\, \bu \in L^2(\Gamma)^3\,:\,\bu_T \in \bV_T,~u_N\in L^2(\Gamma)\,\},  \quad\text{with}~~
\|\bu\|_{V_\ast}^2:=\|\bu_T\|_{1}^2+\tau\|u_N\|^2,
\]
where we introduce parameter $\tau>0$ in the definition of the norm for the convenience of finite element analysis in section~\ref{s_error}.
The following embeddings are obvious:
\[
\bV_T\subset\bV\subset\bV_\ast\subset L^2(\Gamma)^3.
\]
One useful observation is that bilinear forms in \eqref{defblfc}--\eqref{defblfb} are well defined and continuous on the larger space $\bV_\ast$. To see this, one first notes the identity
$\nablaG\bu=\nablaG\bu_T + u_N \bH,$
for any $\bu\in\bV$, where $\bH := \nabla_\Gamma \bn$ is the shape operator (second fundamental form) on $\Gamma$. Hence, using $\bH=\bH^T$ we also get
\begin{equation} \label{idfund}
E_s(\bu)=E_s(\bu_T) + u_N \bH,\quad \div_\Gamma\bu=\div_\Gamma\bu_T + u_N\mbox{tr}(\bH).
\end{equation}
This identity allows to define $\nablaG\bu,E_s(\bu)\in L^2(\Gamma)^{3\times3}$ and $\div_\Gamma\bu\in L^2(\Gamma)$ for all  $\bu\in\bV_\ast$. Moreover, for a $C^2$ surface
$\|\bH\|_{L^\infty(\Gamma)}\le C<\infty$ and one shows with the help of Cauchy--Schwarz inequality,
\begin{align}\label{a_cont}
a(\bu,\bv)\le c\,\|\bu\|_{V_\ast}\|\bv\|_{V_\ast},\quad\bu\in {\bV_\ast},~\bv\in \bV_\ast,\\
\label{b_cont}
b(\bu,q)\le   c\,\|\bu\|_{V_\ast}\|q\|,\quad\bu\in {\bV_\ast},~q\in L^2(\Gamma).
\end{align}
Using $\bH=\bH^T$, $\bH\bn=0$, we also work out for the trilinear form,
\begin{equation} \label{c_cont}
\begin{split}
c(\bu,\bv,\bpsi)&=c(\bu_T,\bv_T,\bpsi_T)+ c(\bu_T,v_N\bn,\bpsi_T)\\
                &\le
                \|\bu_T\|_{L^4(\Gamma)}\|\nabla_\Gamma\bv_T\|_{L^2(\Gamma)}\|\bpsi_T\|_{L^4(\Gamma)}\\
                 &\quad + \|\bu_T\|_{L^4(\Gamma)}\|v_N\|_{L^2(\Gamma)}\|\bpsi_T\|_{L^4(\Gamma)}\\
                &\le
                \|\bu_T\|_{L^4(\Gamma)}\|\bv\|_{V_\ast}\|\bpsi_T\|_{L^4(\Gamma)}\\
                &\le c\|\bu_T\|_{1}\|\bv\|_{V_\ast}\|\bpsi_T\|_{1},
\end{split}
\end{equation}
where for the last inequality we used the embedding $H^1(\Gamma)\subset L^4(\Gamma)$.
A sharper estimate for the $c$-form follows from the  Gagliardo--Nirenberg inequality,
\begin{equation} \label{Ladyzh}
\|\bu_T\|_{L^4(\Gamma)}\le c \|\bu_T\|^{\frac12}\|\bu_T\|_{1}^{\frac12};
\end{equation}
see, e.g., \cite{ledoux2003improved}. We close this section noting that the following infsup condition
for $b(\bv,p)$ form  can be easily shown~\cite{Jankuhn1},
\begin{equation} \label{infsup}
 \sup_{\bv\in{\bV_T^0}}\frac{b(\bv,p)}{\|\bv\|_{1}}  \geq c_0 \|p\| \quad \text{for all}~~p\in L^2_0(\Gamma).
\end{equation}

\section{Finite Element Method}\label{s_TraceFEM}

For the discretization of the variational problem \eqref{weak} we apply the trace finite element approach (TraceFEM) introduced in \cite{ORG09} for elliptic equations on surfaces and extended in \cite{olshanskii2018finite} for the surface steady Stokes problem. The TraceFEM is a geometrically
unfitted discretization technique in a spirit of XFEM and cutFEM. Therefore it allows very flexible treatment of complex and implicitly defined surfaces.
To apply the method, we assume that $\Gamma$ is strictly contained in a polygonal domain $\Omega \subset \R^3$, which is our computational domain.
We consider a family $\{\T_h\}_{h >0}$ of shape regular tetrahedral tessellations  of $\Omega$.
The subset of tetrahedra that have a nonzero intersection with $\Gamma$ is collected in the set
denoted by $\T_h^\Gamma$. For the analysis of the method,  we assume $\{\T_h^\Gamma\}_{h >0}$ to be quasi-uniform with the characteristic mesh size  $h$.

The domain formed by all tetrahedra in $\T_h^\Gamma$ is denoted by $\OGamma$.
On $\T_h^\Gamma$ we use a standard finite element space of continuous functions that are piecewise-polynomial of degree $1$.
This so-called \emph{bulk finite element space} is denoted by $V_h$,
\[
V_h=\{v\in C(\OGamma)\,:\, v\in P_1(T)~\text{for any}~T\in\T_h^{\Gamma}\}.
\]
The numerical approach allows higher order polynomial spaces, and we comment in the text, where modifications are required for this. However, in this paper we analyse and experiment with the $P_1$ spaces.



The velocity and pressure finite element spaces are
\[
\bU_h:= (V_h)^3, \quad Q_h :=V_h\cap L^0_2(\Gamma).
\]
Restriction of a finite element function on $\Gamma$ is an element of $\bV_\ast$, i.e. it does not necessarily satisfy the  $\bu\cdot\bn=0$ condition.
It is not straightforward to build a finite element method which is conformal with respect to this tangentiality condition. As discussed in the introduction, we use a penalty method to enforce the tangentiality condition weakly.

To define the finite element method, we also need an extension $\bn^e$ of the normal vector from $\Gamma$ to $\OGamma$.
We choose $\bn= \nabla d$ in $\OGamma$, where $d$ is the signed distance function to $\Gamma$.
In practice, $d$ is often not available and thus we use approximations. This and other implementation details are reviewed in~section~\ref{s_impl}.
We introduce the following finite element bilinear forms:
\begin{align}
 a_h(\bu,\bv) &= \int_\Gamma E_s( \bu): E_s( \bv)\, ds+\tau \int_{\Gamma} u_N v_N \, ds + \rho_u \int_{\OGamma} (\nabla \bu \bn)\cdot (\nabla  \bv \bn) \, dx, \label{Ah} \\
s_h(p,q)& = \rho_p  \int_{\OGamma} \nabla p\cdot\nabla q \, dx, \label{sh}
\end{align}
with some real parameters $\tau>0$, $\rho_u\ge0$, $\rho_p\ge0$.
The forms are well defined for $p,q \in H^1(\OGamma)$, $\bu,\bv \in H^1(\OGamma)^3$.

Assuming a constant time step $\Delta t=\frac{T}{N}$,  we use the notation $\bu^k(\bx):=\bu(t^k, \bx)$, $t^k=k\Delta t$ and similar
for $p$.
The semi-implicit time discretization and the trace finite element method result in the following scheme:
Given $\bu^{k-1}_h,\bu^{k-2}_h\in\bU_h$, find $(\bu_h^k, p_h^k) \in \bU_h \times Q_h$ solving
\begin{equation} \label{discrete}
 \begin{aligned}
  (\left[\bu_h\right]_t^{k},\bv_h)+ a_h(\bu_h^{k},\bv_h)+ c^\ast(\widetilde\bu_h^{k},\bu_h^{k},\bv_h) + b(\bv_h,p_h^k) & =(\blf^{k},\bv_h)  \\
  b(\bu_h^{k},q_h)-s_h(p_h^{k},q_h) & = 0
 \end{aligned}
\end{equation}
for all  $\bv_h \in \bU_h$ and $q_h \in Q_h$, $k=2,3,\dots,N$. 
In this paper, we consider the second order method with
\begin{equation}\label{BDF2}
\left[\bu_h\right]_t^{k} =\frac{3\bu_h^{k}-4\bu_h^{k-1}+\bu_h^{k-2}}{2\Delta t},\quad \widetilde\bu_h^{k}=2\bu_h^{k-1}- \bu_h^{k-2}.
\end{equation}
For $k=1$, we set $\left[\bu_h\right]_t^{k} =(\bu_h^{k}-\bu_h^{k-1})/{\Delta t}$ and $\widetilde\bu_h^{k}=\bu_h^{k-1}$.
Following~\cite{Temam84} and other work on numerical analysis of incompressible fluid systems, we   explicitly skew-symmetrize the trilinear form,
\begin{equation}\label{formC}
c^\ast(\bw,\bu,\bv)=\frac12(c(\bw,\bu,\bv)-c(\bw,\bv,\bu)).
\end{equation}
Due to identity \eqref{c_skew}, this is a consistent modification.
\smallskip

\begin{remark}[$a_h$-form]\rm The second term in the definition of $a_h$ penalizes the non-zero normal velocity component. The third (volumetric) term is the so-called volume normal derivative stabilization~\cite{burman2016cutb,grande2017higher}.
The term vanishes for the strong solution $\bu$ of equations~\eqref{strongform-1}--\eqref{strongform-2},
since one can always assume a normal extension of $\bu$ off the surface.
The inclusion of this term stabilizes the {resulting algebraic system}.
Indeed, if $\rho_u=0$, then for a natural nodal basis in $\bU_h$, small cuts of the background triangulation by the surface  may lead to arbitrarily small diagonal entries in the resulting matrix. The stabilization term in \eqref{Ah} eliminates this problem since for a suitable choice of $\rho_u$ it
allows to get control over the $L^2(\OGamma)$-norm of $\bv_h\in\bU_h$
by the problem dependent norm
\[
\left(|\Delta t|^{-1}\|\bv_h\|^2+|\bv_h|_a^2
\right)^{\frac12}\quad  \text{with}~~|\bv|_a^2= a_h(\bv,\bv).
\]
We note that other efficient stabilization techniques exist; see~\cite{burman2016cutb} and  the review in~\cite{olshanskii2017trace}.
\end{remark}
\begin{remark}[$s_h$-form]\rm
The bilinear form $s_h$ is introduced for finite element pressure stabilization. In addition to  stabilizing the nodal basis with respect to small element cuts, $s_h$-term also stabilizes the velocity--pressure pair against the violation of the inf-sup condition (the discrete counterpart of \eqref{infsup}). For this reason, both tangential and normal components of the pressure gradient (which together form the full gradient) are included in the definition of $s_h$.
For $P_1$--$P_1$ bulk finite elements used in this paper, the stabilization resembles
the well-known  Brezzi--Pitk\"{a}ranta stabilization~\cite{BP} for the planar Stokes problem. For higher order
elements, the pressure stabilization should be updated to preserve higher order accuracy. One way of doing this is to split
between normal and inf-sup (pressure--velocity) stabilizations
\begin{equation}\label{form_s}
s_h(p,q) = \rho_{p,1}\int_{\OGamma}  \frac{\partial p}{\partial\bn} \frac{\partial q}{\partial\bn} \, dx +  \rho_{p,2}  \int_{\OGamma} \kappa_h(\nablaG p)\,\kappa_h(\nablaG q) \, dx,
\end{equation}
where $\kappa_h$ is a suitably defined elementwise `fluctuation' operator; see, e.g., \cite{ganesan2008local} for the planar case. This or other possible ways to stabilize the  method for higher order finite element pairs will be studied elsewhere.
\end{remark}

Following the analysis for the surface Stokes problem~\cite{olshanskii2018finite}, we set
\begin{equation} \label{parameters}
\rho_p\simeq\rho_u\simeq h,
\end{equation}
which is a minimal possible stabilization from a wide range of acceptable parameters; see, \cite{burman2016cutb,grande2017higher} for the analysis of the normal stabilization for scalar problems.
We write $x\lesssim y$ to state that the inequality  $x\le c y$
holds for quantities $x,y$ with a constant $c$, which is \textit{independent of
$h$ and the position of $\Gamma$ over the background mesh}.
Similarly for $x\gtrsim y$, and $x\simeq y$
will mean that both $x\lesssim y$ and $x\gtrsim y$ hold.

\subsection{Implementation details} \label{s_impl}
 \rm We discuss some implementation aspects of the trace finite element discretization \eqref{discrete}. In the bilinear forms $a_h(\cdot,\cdot)$, $c(\cdot,\cdot,\cdot)$ full gradients of the arguments are computed and next projection $\bP$ is applied.  These can be computed as in standard finite element methods. It is important for the implementation that in $a_h(\cdot,\cdot)$ and $c^\ast(\cdot,\cdot,\cdot)$  we do \emph{not} need derivatives of projected velocities, e.g. of $(\bu_h)_T$.  To avoid differentiation of $\bP$ in the $b$-form, we rewrite the bilinear form as $b(\bv_h,p_h) =  \int_\Gamma \nabla_\Gamma p_h \cdot \bv_h \, ds= \int_\Gamma (\bP \nabla p_h )\cdot \bv_h \, ds  $. This differentiation by parts is valid for $H^1$-conforming pressure finite element spaces, as used in this paper.  Implementation then only requires an approximation of $\bn_h \approx \bn$ and not of derivatives of $\bn$.

In the implementation of this method one typically replaces $\Gamma$
 by an approximation $\Gamma_h \approx \Gamma$ such that integrals over $\Gamma_h$
 can be efficiently computed. Furthermore, the exact normal $\bn$ is approximated by  $\bn_h \approx \bn $.
 In the literature on finite element methods for surface PDEs, this is standard practice.
 We will use a piecewise planar surface approximation $\Gamma_h$ with ${\rm dist}(\Gamma,\Gamma_h) \lesssim h^2$.
 If one is interested in surface FEM with higher order surface approximation, we refer to the recent paper \cite{grande2017higher}.
We assume a level set representation of $\Gamma$:
 \[
 \Gamma=\{\bx\in\mathbb{R}^3\,:\,\phi(\bx)=0\},
 \]
with some smooth function $\phi$ such that $|\nabla\phi|\ge c_0>0$ in a neighborhood of $\Gamma$.
 For the numerical experiments in section~\ref{s_num} we use a piecewise planar surface approximation:
\[
\Gamma_h=\{\bx\in\mathbb{R}^3\,:\, I_h(\phi(\bx))=0\},
\]
where $I_h(\phi(\bx))\in V_h$ is the nodal interpolant of $\phi$.
As for the  construction of suitable normal approximations $\bn_h \approx \bn$,
 several techniques are available in the literature.
 One possibility is to use $\bn_h(\bx)=\frac{\nabla \phi_h(\bx)}{\|\nabla \phi_h(\bx)\|_2}$,
 where $\phi_h$ is a finite element approximation of a level set function $\phi$ which characterizes $\Gamma$.
 This is technique we use in section~\ref{s_num}, where $\phi_h$ is defined as a $P_2$ nodal interpolant of $\phi$.
  Analyzing the effect of resulting geometric errors is beyond the scope of this paper.

\section{Analysis}\label{s_error}
In this section, we present stability and error analysis of the finite element method \eqref{discrete}.
We allow non-zero right hand side in the discrete incompressibility condition, i.e., we consider
\begin{equation}\label{discrete2}
 b(\bu_h^{k},q_h)-s_h(p_h^{k},q_h)  = g^k(q_h)
\end{equation}
instead of the second equation in \eqref{discrete}, where $g^k$ is a functional on $Q_h$. We need this generalization to properly handle certain consistency terms in the error analysis. For the analysis we also assume
\begin{equation}\label{penalty}
  \tau\gtrsim 1.
\end{equation}

\subsection{Numerical stability}
For the energy balance of the finite element method, we test \eqref{discrete} with $\bv_h=\bu_h^k$, $q_h=-p_h^k$. To handle the discrete time derivative~\eqref{BDF2}, we use the following polarization identity:
\[
4\Delta t(\left[\bu_h\right]_t^{k},\bu_h^k)=\|\bu_h^k\|^2+\|\widetilde\bu_h^{k+1}\|^2-(\|\bu_h^{k-1}\|^2+\|\widetilde\bu_h^{k}\|^2)
+|\Delta t|^4\|\left[\bu_h\right]_{tt}^{k-1}\|^2,
\]
with $\left[\bu_h\right]_{tt}^{k}=(\bu_h^{k+1}-2\bu_h^k+\bu_h^{k-1})/|\Delta t|^2$.
After simple calculations we get for $k=2,3,\dots,N$,
\begin{multline}\label{balance_h}
\frac1{4\Delta t}(\|\bu_h^k\|^2+\|\widetilde\bu_h^{k+1}\|^2)+\nu\|E_s(\bu_h^k)\|^2+\tau\|u_{h,N}^k\|^2 \\ +\underbrace{\frac{|\Delta t|^3}4\|\left[\bu_h\right]_{tt}^{k-1}\|^2+\rho_u\|(\bn \cdot \nabla) \bu_h^k\|_{L^2(\OGamma)}^2+\rho_p \|\nabla p_h^k\|_{L^2(\OGamma)}^2}_{\text{$O(\Delta t)$ and $O(h^2)$ dissipative terms}}\\=\frac1{4\Delta t}(\|\bu_h^{k-1}\|^2+\|\widetilde\bu_h^{k}\|^2)+(\blf^k,\bu_h^k)-g^k(p_h^k).
\end{multline}
An analogous equality with obvious modifications holds for $k=1$.
The discrete balance \eqref{balance_h} resembles \eqref{balance_eq} up to several dissipative terms. Note that the true solution of \eqref{strongform-1}--\eqref{strongform-2} is tangential to $\Gamma$ and there is no $u_N$-terms in \eqref{balance_eq}. For the finite element solution, we further show that the term $\tau\sum_{k=1}^N\Delta t\|u_{h,N}^k\|^2$ is of order $O(|\Delta t|^4+h^2+\tau^{-1}+h^4\tau)$.  So its contribution to energy dissipation is of the second order in space and time for the penalty parameter of order $h^{-2}$.
Other dissipative terms, which are not present in \eqref{balance_eq} (middle line  of \eqref{balance_h}), result from time stepping and stabilization procedures. There order with respect to discretization parameters is $O(|\Delta t|^3)$ for the first term 
and $O(h^2)$ for the second and the third (to see this, note \eqref{parameters} and the extra scaling $O(h)$ resulting from the integration over the thin strip $\OGamma$).

To handle the source term $g^k(p_h^k)$ on the right hand side of \eqref{balance_h}, we need the dual norm to the one induced by the pressure stabilization term:
\[
\|g\|_{-s}:=\sup_{q\in Q_h}g(q)/\|q\|_{s},\quad\text{for}~g\in Q_h',\quad  \|q\|_{s}=s_h(q,q)^{\frac12};
\]
then it obviously holds,
$|g^k(p_h^k)|\le \|p_h^k\|_s\|g^k\|_{-s}\le \frac{1}{2}(\|p_h^k\|_s^2+\|g^k\|_{-s}^2).$
In the same way, we treat the forcing term $|(\blf^k,\bv_h^k)|\le \|\blf^k\|_{\bV_\ast'}\|\bv_h^k\|_{\bV_\ast}$. Note that we need a norm on the larger space $\bV_\ast$ for the analysis of the discrete problem comparing to the energy estimates in section~\ref{s_weak}. As a consequence of the Korn inequality \eqref{korn1} and \eqref{penalty}, the $\bV_\ast$ norm is controlled by the problem dependent norm:
\begin{equation}\label{korn_dis}
\|\bv\|_{\bV_\ast}\lesssim (|\bv|_a^2+\|\bv\|^2)^\frac12,\quad\text{for all}~\bv\in\bV_\ast.
\end{equation}

Multiplying \eqref{balance_h} by $4\Delta t$ and summing up for $k=1,\dots,n$ and treating the $g^k$ and $\blf^k$ terms as above we arrive on the following  estimate
\begin{equation}\label{energy_est1}
\|\bu_h^n\|^2+\sum_{k=1}^{n}\Delta t\left\{|\bu_h^k|_a^2+\|p_h^k\|_{s}^2\right\} \lesssim
\|\bu_h^0\|^2+\sum_{k=1}^{n}\Delta t\left\{\|\blf^k\|_{\bV_\ast'}\|\bu_h^k\|_{\bV_\ast}+\|g^k\|_{-s}^2\right\}.
\end{equation}
To estimate the $\bV_\ast$-- norms of $\bu_h^k$, we proceed as in the continuous case of section~\ref{s_weak}, with the only change that instead of \eqref{korn1} we use \eqref{korn_dis} and summation $\sum_{k=1}^{n}\Delta t$ in place of $\int_0^t$. These arguments lead to the estimate
\begin{equation}\label{energy_est2}
\sum_{k=1}^{n}\Delta t\|\bu_h^k\|_{\bV_\ast}^2 \lesssim
\|\bu_h^0\|^2+\sum_{k=1}^{n}\Delta t\left\{\|\blf^k\|_{\bV_\ast'}^2+\|g^k\|_{-s}^2\right\}.
\end{equation}
Next we apply the Cauchy--Schwarz inequality to the $\blf^k$-term in \eqref{energy_est1},
\begin{equation}\label{energy_est3}
\sum_{k=1}^{n}\Delta t\|\blf^k\|_{\bV_\ast'}\|\bu_h^k\|_{\bV_\ast}\le \sum_{k=1}^{n}\Delta t\|\blf^k\|_{\bV_\ast'}^2+\sum_{k=1}^{n}\Delta t\|\bu_h^k\|_{\bV_\ast}^2,
\end{equation}
and use \eqref{energy_est2} to estimate the second term on the right hand side. Thus \eqref{energy_est1}--\eqref{energy_est3} lead to our final numerical  stability estimate
\begin{equation}\label{energy_est}
\|\bu_h^n\|^2+\sum_{k=1}^{n}\Delta t\left\{\|\bu_h^k\|_{\bV_\ast}^2+\|p_h^k\|_{s}^2\right\} \lesssim
\|\bu_h^0\|^2+\sum_{k=1}^{n}\Delta t\left\{\|\blf^k\|_{\bV_\ast'}^2+\|g^k\|_{-s}^2\right\},
\end{equation}
for $ n=1,2,\dots,N.$
We note that the norm $\|g^k\|_{-s}$ is mesh-dependent through  parameter $\rho_p$ in the definition of $s_h$ form. We admit the presence of such term on the right-hand side of the stability estimate for the following reason: For incompressible surface fluids, either we have $g^k=0$ or we apply \eqref{energy_est} for an equation with  a consistency term $g^k$, which scales with $h$ in a suitable way. Next, we analyse convergence of the method.  We start with consistency estimates.

\subsection{Consistency estimates} \label{sec:consist}
Further we need $\Gamma\in C^3$ assumption, since we deal with normal extension of functions from $\Gamma$ to $\OGamma$ and we need the extended normal vector field to be at least from $C^2(\OGamma)$.
For the normal extension of a sufficiently smooth function $v$ defined on $\Gamma$,  the following estimates will be useful \cite{ORG09,reusken2015analysis}:
 \begin{equation}\label{normal}
\begin{split}
 h^\frac{1}{2}\|\nablaG v\|& \simeq \|\nabla v\|_{L^2(\OGamma)}, \quad \text{for all}~~v \in H^1(\Gamma), \\
h^{\frac12}\|v\|& \simeq \|v\|_{L^2(\OGamma)}, \quad \text{for all}~~v \in L^2(\Gamma),\\
  h^{\frac12} \|v\|_{H^2(\Gamma)}  &\gtrsim \|v\|_{H^2(\OGamma)} , \quad \text{for all}~~v \in H^2(\Gamma).
\end{split}
 \end{equation}
Applying the first estimate in \eqref{normal} componentwise and using that normal derivatives vanish, we also get for all $\bv \in H^1(\Gamma)^3$:
  \begin{equation}\label{normalv}
  \|\nabla\bv\|_{L^2(\OGamma)}\lesssim h^{\frac12}\|\bv\|_1.
 \end{equation}

Recall the notation $\bu^{k}=\bu(t^k)$, $p^k=p(t^k)$.
Testing \eqref{weak} with $\bpsi=\bP\bv_h|_\Gamma$ for $\bv_h\in\bV_h$ and $q=q_h$ for $q_h\in Q_h$,  we find that  $\bu$, $q$ satisfy
\begin{equation} \label{e:exid}
 \begin{aligned}
  (\left[\bu\right]_t^{k},\bv_h)+ a_h(\bu^{k},\bv_h)+ c^\ast(\widetilde\bu^{k},\bu^{k},\bv_h) + b(\bv_h,p^k) & = (\blf^k,\bv_h) +\text{consist}^k_u(\bv_h) \\
  b(\bu^{k},q_h)-s_h(p^{k},q_h) & =\text{consist}^k_p(q_h)
 \end{aligned}
\end{equation}
for all  $\bv_h \in \bU_h$ and $q_h \in Q_h$ with
\begin{align*}
 \text{consist}^k_u(v_h) :=
  & \hphantom{+}
    (\left[\bu\right]_t^{k}-\bu_t(t^k), \bv_h)
  +   c^\ast(\widetilde\bu^{k}-\bu^{k},\bu^{k},\bv_{h})+a_h(\bu^{k},v_{h,N}\bn)\\ &- \frac12c(\bu^{k},v_{h,N}\bn,\bu^{k})
    \\
\text{consist}^k_p(q_h):=
  &\hphantom{+} -s_h(p^{k},q_h).
\end{align*}
Note that $\bP\bv_h$ yields to $\bv_h$ in some forms, since $\bu_t(t^k)$, $\left[\bu\right]_t^{k}$ and $\blf^k$ are tangential to $\Gamma$ and $c(\bu^{k},\bu^{k},v_{h,N}\bn)=0$ holds. 

For further analysis we need certain regularity for the solution to the surface Navier--Stokes system.
\begin{assumption}\label{A1}
The solution of \eqref{strongform-1}--\eqref{strongform-2} is such that
\begin{equation}\label{eqA1}
\begin{split}
\bu&\in L^\infty(0,T;H^{2}(\Gamma)^3),\quad
p\in L^\infty(0,T;H^1(\Gamma))\\ \frac{d^i\bu}{dt^i}&\in L^\infty(0,T;H^{1}(\Gamma)^3),~i=1,2,\quad \frac{d^3\bu}{dt^3}\in L^\infty(0,T;L^2(\Gamma)^3).
\end{split}
\end{equation}
\end{assumption}

\begin{lemma}\label{l_consist} Assume \eqref{eqA1},
then the consistency error has the bound
\begin{equation}\label{est:consist}
  |\text{\rm consist}^k_u(\bv_h)|\lesssim (|\Delta t|^2+\tau^{-\frac12})\,\|\bv_{h}\|_{V_\ast},
  \quad |\text{\rm consist}^k_p(q_h)| \lesssim h \|q_h\|_s,~ k\ge2.
\end{equation}
\end{lemma}
\begin{proof} We treat $\textrm{consist}(\bv_h)$ term by term:
\begin{equation*}
\begin{split}
|(\bu_t(t^k)&-\left[\bu\right]_t^{k}), \bv_h)| \\&=  \left|\int_{\Gamma}\left(\int_{t^{k-2}}^{t^k}\frac{(t-t^{k-2})^2}{4\Delta t} \bu_{ttt} \,dt
-\int_{t^{k-1}}^{t^k}\frac{(t-t^{k-1})^2}{\Delta t}  \bu_{ttt} \,dt\right)\cdot\bv_h\, dx\right|\\
& \lesssim |\Delta t|^2\sup_{t\in[t^{k-2},t^k]}\| \bu_{ttt}\|\|\bv_{h}\|\\
&\lesssim
|\Delta t|^2\| \bu_{ttt}\|_{L^{\infty}(0,T;L^2(\Gamma))}\|\bv_{h}\|_{V_\ast}.
 \end{split}
\end{equation*}
Using the definition of the trilinear form, identity $\nablaG\bu=\nablaG\bu_T + u_N \bH,$ and estimates \eqref{c_cont},
we have
\begin{equation*}
\begin{split}
c^\ast(\widetilde\bu^{k}-\bu^{k},\bu^{k},\bv_{h})&=\frac12( c(\widetilde\bu^{k}-\bu^{k},\bu^{k},\bv_{h})-c(\widetilde\bu^{k}-\bu^{k},\bv_{h},\bu^{k}))\\
&=\frac12\left( 2 c(\widetilde\bu^{k}-\bu^{k},\bu^{k},\bv_{h,T})+\int_{\Gamma}v_{h,N} (\widetilde\bu^{k}-\bu^{k})^T\bH\bu^{k}\,dx\right)\\
&\lesssim |\Delta t|^2 \sup_{t\in[t^{k-2},t^k]}\|\bu_{tt}\|_{L^4(\Gamma)}( \|\nabla_\Gamma\bu^k\|_{L^4(\Gamma)} \|\bv_{h,T}\|+\|v_{h,N}\| \|\bu^{k}\|_{L^4(\Gamma)})\\
&\lesssim |\Delta t|^2\|\bu_{tt}\|_{L^\infty(0,T;L^4(\Gamma))} \|\nabla_\Gamma\bu\|_{L^\infty(0,T;L^4(\Gamma))} \|\bv_{h}\|\\
&\lesssim |\Delta t|^2\|\bu_{tt}\|_{L^\infty(0,T;H^1(\Gamma))}\|\bu\|_{L^\infty(0,T;H^2(\Gamma))} \|\bv_{h}\|_{V_\ast}.
  \end{split}
\end{equation*}
In the last inequality we use embedding $H^1(\Gamma)\subset L^4(\Gamma)$ and \eqref{penalty}. Further, we compute
\begin{equation*}
\begin{split}
|a_h(\bu^{k},v_{h,N}\bn)|&=\left|\int_\Gamma E_s(\bu^{k}):\bH v_{h,N}\,ds\right|\lesssim \|\bu^{k}\|_1\|v_{h,N}\|\lesssim \tau^{-\frac12}\|\bu^{k}\|_1\|\bv_{h}\|_{V_\ast},\\
|c(\bu^{k},v_{h,N}\bn,\bu^{k})|& =\left|\int_\Gamma (\,(\bu^{k})^T\bH\bu^{k}) v_{h,N}\,ds\right|\lesssim \|\bu^{k}\|_{L^4(\Gamma)}^2 \|v_{h,N}\|\\
&\lesssim\|\bu^{k}\|_{1}^2 \|v_{h,N}\|\lesssim \tau^{-\frac12}\|\bu^{k}\|_{1}^2 \|\bv_{h}\|_{V_\ast}.
  \end{split}
\end{equation*}
For the second consistency term we have thanks to the definition of the $s$-norm, \eqref{parameters} and \eqref{normal}:
\[
\begin{split}
|\text{consist}^k_p(q_h)|&\le \|p^k\|_s\|q_h\|_s\lesssim \sqrt{h} \|\nabla p^k\|_{L^2(\OGamma)} \|q_h\|_s
\lesssim h \|\nabla_\Gamma p^k\| \|q_h\|_s\\
& \lesssim h \| p\|_{L^\infty(0,T;H^1(\Gamma))} \|q_h\|_s.
  \end{split}
\]
Now \eqref{est:consist} follows from the assumptions \eqref{eqA1} on the regularity of $\bu$ and $p$.
\end{proof}

For $k=1$ estimate as in \eqref{est:consist} holds with $|\Delta t|^2$ replaced by $|\Delta t|$.

Let $\err_u^k = \bu^k - \bu_h^k$, $\err_p^k=p^k-p_h^k$, subtracting \eqref{discrete} from \eqref{e:exid} we obtain the error equations
\begin{equation} \label{e:erreq}
 \begin{aligned}
  (\left[\err_u\right]_t^{k},\bv_h)+ a_h(\err^{k}_u,\bv_h)+ c^\ast(\widetilde\bu^{k}_h,\err^{k}_u,\bv_h) + b(\bv_h,\err_p^k) & =  \text{consist}^k_u(\bv_h)-c^\ast(\widetilde\err^{k}_u,\bu^{k},\bv_h), \\
  b(\err_u^{k},q_h)-s_h(\err_p^{k},q_h) & =\text{consist}^k_p(q_h),
 \end{aligned}
\end{equation}
for all  $\bv_h \in \bU_h$ and $q_h \in Q_h$.

\subsection{Error estimate in the energy norm} \label{sec:aprioriest}
 We let $\bu^k_I=\mathcal{I} (\bu^k) \in \bU_h$ and $p^k_I=\mathcal{I} (p^k) \in Q_h^k$ be the Lagrange interpolants  for (extensions of) $\bu^k$ and $p^k$ in $\OGamma$; we assume both surface velocity and pressure to be  sufficiently smooth so that the interpolation is well-defined.
The following approximation properties of  $\bu^k_I$ and $p^k_I$ are well-known from the literature; see, e.g,
\cite{ORG09,reusken2015analysis,olshanskii2017trace}:
\begin{equation}\label{interp_est}
 \begin{aligned}
 \|\bu - \bu^k_I\|+h\|\nablaG(\bu - \bu^k_I)\|+h^{\frac12}\|\nabla(\bu - \bu^k_I)\|_{L^2(\OGamma)}&\le C h^2\|\bu\|_{H^2(\Gamma)},\\
 \|p - p^k_I\|+h\|\nablaG(p - p^k_I)\|+h^{\frac12}\|\nabla(p - p^k_I)\|_{L^2(\OGamma)}&\le C h\|p\|_{H^1(\Gamma)}.
  \end{aligned}
\end{equation}
We emphasize  that a constant $C$ in \eqref{interp_est} depends only  on the shape regularity of tetrahedra from $\OGamma$, but not on how $\Gamma$ intersects them.

Following the standard line of arguments, we split the error into finite element and approximation parts,
\[
\err_u^k=\underset{\mbox{$\be^k$}}{\underbrace{(\bu^k-\bu^k_I)}}+\underset{\mbox{$\be^k_h \in \bV_h$}}{\underbrace{(\bu^k_I-\bu^k_h)}}, \quad
\err_p^k=\underset{\mbox{$e^k$}}{\underbrace{(p^k-p^k_I)}}+\underset{\mbox{$e^k_h \in Q_h$}}{\underbrace{(p^k_I-p^k_h)}}.
\]
Equation \eqref{e:erreq} yields
\begin{equation}\label{e:err1}
 \begin{aligned}
  (\left[\be_h\right]_t^{k},\bv_h)+ a_h(\be^{k}_h,\bv_h)&+ c^\ast(\widetilde\bu^{k}_h,\be^{k}_h,\bv_h) + b(\bv_h,e^k_h)\\&  = \text{consist}^k_u(\bv_h)- \text{interpol}^k_u(\bv_h)-c^\ast(\widetilde\err^{k}_u,\bu^{k},\bv_h),\\
    b(\be_h^{k},q_h)-s_h(e^{k}_h,q_h) & =\text{consist}^k_p(q_h)-\text{interpol}^k_p(q_h),
 \end{aligned}
\end{equation}
for all  $\bv_h \in \bU_h$ and $q_h \in Q_h$,
with the interpolation terms
\[
\begin{aligned}
 \text{interpol}^k_u(\bv_h)&=(\left[\be\right]_t^{k},\bv_h)+ a_h(\be^{k},\bv_h)+ c^\ast(\widetilde\bu^{k}_h,\be^{k},\bv_h) + b(\bv_h,e^k),\\
 \text{interpol}^k_p(q_h)&=b(\be^{k},q_h)-s_h(e^{k},q_h)
\end{aligned}
\]
We  estimate the interpolation and $c^\ast$ terms on the right hand side in the following lemma.

\begin{lemma}\label{l_interp} Assume \eqref{eqA1},
then it holds
\begin{equation}\label{est_inter}
\begin{split}
  |\text{\rm interpol}^k_u(\bv_h)|&\lesssim (h+\tau^{\frac12}h^2+ h\|\widetilde\bu^{k}_h\|_{V_\ast})\, \|\bv_{h}\|_{V_\ast}, \quad
  |\text{\rm interpol}^k_p(q_h)|\lesssim h\|q_h\|_s, \\
   |c^\ast(\widetilde\err^{k}_u,\bu^{k},\bv_h)|&\lesssim (h+\|\widetilde\be_h^{k}\|)\|\bv_{h}\|_{V_\ast}, \quad k=1,2\dots.
\end{split}
\end{equation}
\end{lemma}
\begin{proof}
We extend $\bu^k_I$, $k=1,\dots,N$, for all $t\in[t^{k-1},t^k]$  as the Lagrange interpolant of $u(t)$ in all nodes from $\OGamma$. 
Since $(\bu^k_I)_t$ is the nodal interpolant for $\bu_t$,
we have thanks to \eqref{interp_est} that
$
\|\be_t\|\lesssim h\|\bu_t\|_{H^{1}(\Gamma)}
$
for $t \in [0,T]$.
Let $k\ge2$, with the help of this bound and the Cauchy--Schwarz inequality we treat the first term in ${\rm interpol}^k_u(\bv_h)$,
\begin{equation*}\label{e:err2}
\begin{split}
\left|\int_\Gamma\right.&\left.\left[\be\right]_t^{k}\cdot\bv_h\,ds\right|\le \left(\frac32\left\|\frac{\be^k-\be^{k-1}}{\Delta t}\right\|+\frac12\left\|\frac{\be^{k-1}-\be^{k-2}}{\Delta t}\right\|\right)\|\bv_h\|  \\ &
 =|\Delta t|^{-1} \left(\frac32 \left\|\int^{t^k}_{t^{k-1}} \be_t(s)\,ds\right\|+\frac12 \left\|\int^{t^{k-1}}_{t^{k-2}} \be_t(s)\,ds\right\| \right)\|\bv_h\|\\&
 \le |\Delta t|^{-\frac12}\left(\frac32 \left(\int^{t^k}_{t^{k-1}}\| \be_t(s)\|^2\,ds\right)^{\frac12}+\frac12 \left(\int^{t^{k-1}}_{t^{k-2}}\| \be_t(s)\|^2\,ds\right)^{\frac12}\right)
 \|\bv_h\|
  \\ &
 \lesssim \,h\sup_{t\in[t^{k-2},t^k]} \| \bu_t\|_{H^1(\Gamma)}\|\bv_h\| \lesssim \,h \| \bu_t\|_{L^\infty(0,T;H^{1}(\Gamma)^3)}\|\bv_h\|_{V_\ast}
\end{split}
\end{equation*}
Similar we handle the term with $\left[\be\right]_t^{k}$ for $k=1$.
Other terms in $\text{\rm interpol}^k_u(\bv_h)$ are handled in a straightforward way using the Cauchy-Schwarz inequality, \eqref{c_cont}, \eqref{interp_est} and \eqref{normal}:
\begin{align*}
  a_h(\be^{k},\bv_h) & \lesssim  ( \|\be^k\|_1 + \tau^{\frac12}\|\be^k\| + h^\frac12 \|\be^k\|_{H^1(\OGamma)})\|\bv_h\|_{V_\ast}\\ &
    \lesssim \left((h+\tau^{\frac12}h^2) \|\bu^k\|_{H^2(\Gamma)}+h^{\frac32} \|\bu^k\|_{H^2(\OGamma)} \right)\|\bv_h\|_{V_\ast}\\
  & \lesssim (h+\tau^{\frac12}h^2) \|\bu^k\|_{H^2(\Gamma)} \|\bv_h\|_{V_\ast},\\
   c^\ast(\widetilde\bu^{k}_h,\be^{k},\bv_h)& \lesssim \|\widetilde\bu^{k}_h\|_{V_\ast}\|\be^{k}\|_{1}\|\bv_h\|_{V_\ast}\lesssim h\|\widetilde\bu^{k}_h\|_{V_\ast}\|\bu^k\|_{H^2(\Gamma)}\|\bv_h\|_{V_\ast},\\
   b(\bv_h,e^k)& \lesssim h\|p^k\|_{H^1(\Gamma)}\|\bv_h\|_{V_\ast}.
\end{align*}
For the second interpolation term $\text{\rm interpol}^k_p(q_h)$, we similarly have by \eqref{interp_est} and \eqref{normal}
\begin{align*}
b(\be^{k},q_h)&\lesssim \|\be^{k}\|\|\nabla_\Gamma q_h\| \lesssim h^2\|\bu^k\|_{H^2(\Gamma)}\|\nabla_\Gamma q_h\|  \le h^2\|\bu^k\|_{H^2(\Gamma)}\|\nabla q_h\|\\ & \lesssim h^{\frac32}\|\bu^k\|_{H^2(\Gamma)}\|\nabla q_h\|_{L^2(\OGamma)} \lesssim h\|\bu^k\|_{H^2(\Gamma)}\|q_h\|_s, \\
|s_h(e^{k},q_h)|&\lesssim h\|p^k\|_{H^1(\Gamma)}\|q_h\|_s.
\end{align*}
Note that we used the estimate $\|\nabla q_h\|\lesssim h^{-\frac12}\|\nabla q_h\|_{L^2(\OGamma)}$, which is elementary, since $\nabla q_h$ is constant in each tetrahedra from $\OGamma$; see, e.g., \cite[Lemma~4.3]{reusken2015analysis} for the (more general) estimate.
Finally, we estimate
\begin{align*}
|c^\ast&(\widetilde\err^{k}_u,\bu^{k},\bv_h)|\le |c^\ast(\widetilde\be^{k},\bu^{k},\bv_h)|+ |c^\ast(\widetilde\be^{k}_h,\bu^{k},\bv_h)|\\
&\lesssim\|\widetilde\be^{k}\|_1\|\bu^{k}\|_1\|\bv_h\|_{V_\ast}+ \|\widetilde\be^{k}_h\|\|\nabla_\Gamma\bu^{k}\|_{L^4(\Gamma)}\|\bv_{h,T}\|_{L^4(\Gamma)}+ \|\widetilde\be^{k}_h\|\|\nabla_\Gamma\bv_{h}\| \|\bu^{k}\|_{L^\infty(\Gamma)}\\
&\lesssim h \|\bu^{k}\|_{H^2(\Gamma)}\|\bv_{h}\|_{V_\ast}+\|\widetilde\be_h^{k}\|\|\bu^{k}\|_{H^2(\Gamma)}\|\bv_{h}\|_{V_\ast}.
\end{align*}
In the last inequality we used $L^4(\Gamma)\subset H^1(\Gamma)$, $L^\infty(\Gamma)\subset H^2(\Gamma)$ and (due to $\nabla_\Gamma\bv_{h}=\nabla_\Gamma\bv_{h,T}+v_{h,N}\bH$) $\|\nabla_\Gamma\bv_{h}\|\lesssim\|\bv_{h}\|_{V_\ast}$. We finally obtain the desired estimate \eqref{est_inter} using the assumption on the regularity of the solution: $ \| \bu_t\|_{L^\infty(0,T;H^{1}(\Gamma)^3)}\lesssim1$,  $\|p^k\|_{H^1(\Gamma)}\le \|p\|_{L^\infty(0,T;H^1(\Gamma))}\lesssim 1$, $\|\bu^{k}\|_{H^2(\Gamma)}\le \|\bu\|_{L^\infty(0,T;H^2(\Gamma))}\lesssim1$. \hfill
\end{proof}
\smallskip

We now apply the stability estimate in \eqref{energy_est} to  the error function, satisfying equation \eqref{e:err1},
and we further use the results in Lemmas~\ref{l_consist} and~\ref{l_interp} to estimate the right hand side.
Since the estimate in Lemmas~\ref{l_consist} holds for $k\ge2$, we first obtain that
\begin{multline}\label{aux1}
\|\be_h^n\|^2+\sum_{k=2}^{n}\Delta t\left\{\|\be_h^k\|_{\bV_\ast}^2+\|e_h^k\|_{s}^2\right\}\\ \lesssim
\|\be_h^1\|^2+(1+\tau h^2)h^2+|\Delta t|^4+\tau^{-1} + \Delta t\sum_{k=2}^{n}(\|\widetilde\be_h^k\|^2+h^2\|\bu^{k}_h\|_{V_\ast}^2)\\
\lesssim
\|\be_h^1\|^2+(1+\tau h^2)h^2+|\Delta t|^4 +\tau^{-1} + \Delta t\sum_{k=0}^{n-1}\|\be_h^k\|^2.
\end{multline}
For the last inequality, we applied the stability bound \eqref{energy_est2} for the finite element solution to conclude that $\Delta t\sum_{k=2}^{n}\|\bu^{k}_h\|_{V_\ast}^2\lesssim1$.

On the first time step of \eqref{discrete}, BDF1 method is applied and instead of the first consistency bound in \eqref{est:consist} we have
$|\text{\rm consist}^1_u(\bv_h)|\lesssim (|\Delta t|+\tau^{-\frac12})\,\|\bv_{h}\|_{V_\ast}$. Moreover, an examination of the proof of Lemma~\ref{l_consist} reveals that this estimate can be improved to
$|\text{\rm consist}^1_u(\bv_h)|\lesssim (|\Delta t|\|\bv_{h}\|+\|v_{h,N}\|)$.  All other estimates of consistency and interpolation terms in Lemmas~\ref{l_consist} and~\ref{l_interp} remain the same. Using this in \eqref{e:err1} for $k=1$ together with $\bv_h=\be_h^1$, $q_h=e^k_h$ leads after simple calculations to
\[
\|\be_h^1\|^2 + \Delta t(|\be_h^1|_{a}^2+\|e_h^1\|_{s}^2) \le \|\be^0_h\|^2+(1+\tau h^2)h^2+|\Delta t|^4 +\tau^{-1}.
\]
Substituting the above inequality to \eqref{aux1} and noting that $\|\be^0_h\|=0$ and $\|\be_h^1\|^2 + \Delta t|\be_h^1|_{a}^2\gtrsim \Delta t\|\be_h^k\|_{\bV_\ast}^2$, we get
\[
\|\be_h^n\|^2+\sum_{k=1}^{n}\Delta t\left\{\|\be_h^k\|_{\bV_\ast}^2+\|e_h^k\|_{s}^2\right\}
\lesssim (1+\tau h^2)h^2+|\Delta t|^4 +\tau^{-1} + \Delta t\sum_{k=0}^{n-1}\|\be_h^k\|^2.
\]
We next apply the discrete Gronwall inequality to obtain
\[ 
\|\be_h^n\|^2+\sum_{k=1}^{n}\Delta t\left\{\|\be_h^k\|_{\bV_\ast}^2+\|e_h^k\|_{s}^2\right\} \lesssim
(1+\tau h^2)h^2+|\Delta t|^4+\tau^{-1}.
\] 
The triangle inequality and approximation properties  \eqref{interp_est} lead to the final error bound:
\begin{equation}\label{error}
\|\err_u^n\|^2+\sum_{k=1}^{n}\Delta t\left\{\|\err_u^k\|_{\bV_\ast}^2+\|\err_p^k\|_{s}^2\right\}
\lesssim
(1+\tau h^2)h^2+|\Delta t|^4+\tau^{-1},
\end{equation}
for $n=1,\dots,N$.

 The main result is summarized in the following theorem.
 \smallskip

\begin{theorem}\label{Th1} Assume $\Gamma\in C^3$ and the solution to the surface fluid system \eqref{strongform-1}--\eqref{strongform-2} is sufficiently smooth such that \eqref{eqA1} holds. For the trace finite element method \eqref{discrete} assume that the background mesh is quasi-uniform, and parameters satisfy \eqref{parameters}, \eqref{penalty}.
Then the finite element method is stable   and the error estimate  \eqref{error} holds.
\end{theorem}
\smallskip

From the result in \eqref{error} we see that the optimal penalty parameter $\tau$ scales with $h^{-2}$. This is consistent with the analysis of the steady surface Stokes and vector Laplacian problems in \cite{olshanskii2018finite,gross2017trace}. Note that the squared $L^2(\Gamma)$-norm of normal component $u_{h,N}^k$ on the left hand side of \eqref{error} is multiplied by $\tau$, which leads to the estimate $\sum_{k=1}^{n}\Delta t\|u_{h,N}^k\|^2\lesssim (\tau^{-1}+h^2)h^2+|\Delta t|^4\tau^{-1}+\tau^{-2}$. This and \eqref{error} yields the following corollary.

\begin{corollary}\label{cor1} Let $\tau\simeq h^{-2}$. Under assumption of Theorem~\ref{Th1} the following error estimate holds:
 \begin{equation}\label{error1}
 \begin{split}
\max_{k=1,\dots,N}\|\err_u^k\|^2+\sum_{k=1}^{N}\Delta t\left\{\|\err_u^k\|_1^2+\|\err_p^k\|_{s}^2\right\}
&\lesssim h^2+|\Delta t|^4,\\
\sum_{k=1}^{N}\Delta t\|u_{h,N}^k\|^2&\lesssim  h^2(h^2+|\Delta t|^4).
\end{split}
\end{equation}
\end{corollary}

\section{Numerical results}\label{s_num}
The section collects results of several numerical experiments that illustrate the performance of the finite element method on a model example of the Navier--Stokes equations posed on a unit sphere embedded in a cubic computational domain $\Omega=[-5/3,5/3]^3$.
We examine the accuracy of the method by varying discretization and penalty parameters. All results agree well with the error bound in Theorem~\ref{Th1} and Corollary~\ref{cor1}. In addition, we include an example which studies energy conservation property of the method.

In all experiments we build a family of unfitted triangulations $\T_{h_\ell}$ of $\Omega$ consisting of $n_\ell^3$ sub-cubes, where each of the sub-cubes is further refined into 6 tetrahedra.
Here $\ell\in\Bbb{N}$ denotes the level of refinement, with mesh size $h_\ell= \frac{10/3}{n_\ell}$ and $n_\ell= 2^{\ell+1}$.
We set parameters $\rho_p=\rho_u=h$, which is in agreement with~\eqref{parameters}. In all experiments BDF2 discretization of the time derivative as in \eqref{BDF2} is used.
To perform numerical integration, we consider nodal interpolant $I_h(\phi)\in V_h$  of the level set function $\phi(\bx) = \|\bx\|_2 -1$, $\bx = (x_1, x_2, x_3)^T$. Further all integrals were computed over $\Gamma_h$, which is the  zero level of $I_h(\phi)$. All implementations were done in DROPS software package~\cite{DROPS}.

\subsection{Convergence to exact smooth solution}\label{s_num1}

\begin{figure}
	\begin{center}
		\begin{tikzpicture}[scale=0.8]
		\def\vara{10}
		\def\varb{1e-1}
		\begin{semilogyaxis}[ xlabel={Refinement level $\ell$},xmax=5,xmin=0, ylabel={Error}, ymin=5E-4
		, ymax=10, legend style={ cells={anchor=west}, legend pos=outer north east} ]
		
		\addplot+[red,mark=o,solid,mark size=3pt,line width=1.5pt] table[x=level, y=t_max_uT] {./test17_h.dat};
		
		\addplot+[blue,mark=+,mark size=3pt,line width=1.5pt] table[x=level, y=t_H1] {./test17_h.dat};
		
		\addplot+[black!30!yellow,mark=diamond,mark size=3pt,line width=1.5pt] table[x=level, y=t_L2_uN] {./test17_h.dat};
		
		\addplot+[black!30!green,mark=x,mark size=3pt,line width=1.5pt] table[x=level, y=t_L2_p] {./test17_h.dat};
		
		\addplot[dashed,line width=1pt] coordinates { 
			(0,\vara) (1,\vara*0.5) (2,\vara*0.25) (3,\vara*0.125) (4,\vara*0.0625) (5,\vara*0.03125) (6,\vara*0.03125/2)
		};
		\addplot[dotted,line width=1pt] coordinates { 
			(0,\varb) (1,\varb*0.5*0.5) (2,\varb*0.25*0.25) (3,\varb*0.125*0.125) (4,\varb*0.0625*0.0625) (5,\varb*0.03125*0.03125) (6,\varb*0.03125*0.03125*0.5*0.5)
		};
		\legend{
			$\max_k\|\bu^k - \bu_h^k\| $,
			$\left(\sum{}_k\, \Delta{}t\|\bu^k -\bu_h^k\|^2_{1}\right)^{1/2} $,
			$\left(\sum{}_k\,\Delta{}t\|\bu_h \cdot \bn_h \|^2\right)^{1/2} $,
			$\left(\sum{}_k\,\Delta{}t\|\bp^k - \bp_h^k\|^2\right)^{1/2} $,
			$\mathcal{O}(h^1)$,
			$\mathcal{O}(h^{2})$}
		\end{semilogyaxis}
		\end{tikzpicture}
		\caption{Velocity and pressure error in various norms against the refinement level $\ell$. Results were computed with $\Delta{}t=2^{1-\ell}/10=O(h)$, $\tau=h^{-2}_\ell$ and $\nu=1$.
		}
		\label{convergence}
	\end{center}
\end{figure}
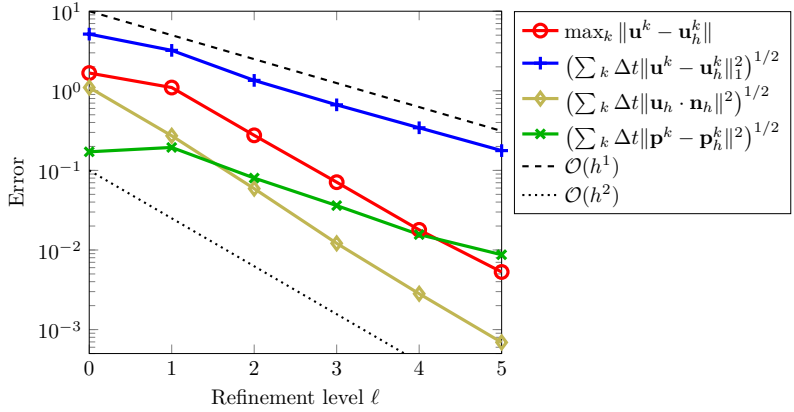

\begin{figure}
	\begin{center}
		\begin{tikzpicture}[scale=0.8]
		\def\vara{10}
		\def\varb{1e-1}
		\begin{semilogyaxis}[ xlabel={Refinement level $\ell$},xmax=5,xmin=0, ylabel={Error}, ymin=5E-4
		, ymax=10, legend style={ cells={anchor=west}, legend pos=outer north east} ]
		
		\addplot+[red,mark=o,solid,mark size=3pt,line width=1.5pt] table[x=level, y=t_max_uT] {./test17_h_nu=0.01.dat};
		
		\addplot+[blue,mark=+,mark size=3pt,line width=1.5pt] table[x=level, y=t_H1] {./test17_h_nu=0.01.dat};
		
		\addplot+[black!30!yellow,mark=diamond,mark size=3pt,line width=1.5pt] table[x=level, y=t_L2_uN] {./test17_h_nu=0.01.dat};
		
		\addplot+[black!30!green,mark=x,mark size=3pt,line width=1.5pt] table[x=level, y=t_L2_p] {./test17_h_nu=0.01.dat};
		
		\addplot[dashed,line width=1pt] coordinates { 
			(0,\vara) (1,\vara*0.5) (2,\vara*0.25) (3,\vara*0.125) (4,\vara*0.0625) (5,\vara*0.03125) (6,\vara*0.03125/2)
		};
		\addplot[dotted,line width=1pt] coordinates { 
			(0,\varb) (1,\varb*0.5*0.5) (2,\varb*0.25*0.25) (3,\varb*0.125*0.125) (4,\varb*0.0625*0.0625) (5,\varb*0.03125*0.03125) (6,\varb*0.03125*0.03125*0.5*0.5)
		};
		\legend{
			$\max_k\|\bu^k - \bu_h^k\| $,
			$\left(\sum{}_k\, \Delta{}t\|\bu^k -\bu_h^k\|^2_{1}\right)^{1/2} $,
			$\left(\sum{}_k\,\Delta{}t\|\bu_h \cdot \bn_h \|^2\right)^{1/2} $,
			$\left(\sum{}_k\,\Delta{}t\|p^k - p_h^k\|^2\right)^{1/2} $,
			$\mathcal{O}(h^1)$,
			$\mathcal{O}(h^{2})$}
		\end{semilogyaxis}
		\end{tikzpicture}
		\caption{Velocity and pressure error in various norms against the refinement level $\ell$. Results were computed with $\Delta{}t=2^{1-\ell}/10=O(h)$, $\tau=h^{-2}_\ell$ and $\nu=0.01$
		}
		\label{convergence1}
	\end{center}
\end{figure}
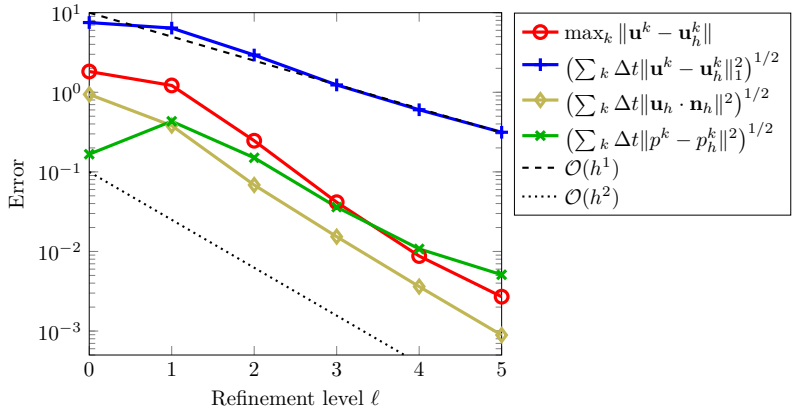

We first test the convergence of the finite element method \eqref{discrete} for the example of a synthetic smooth time-dependent solution. For the exact solution, we define
\begin{equation} \label{exact1}
\bu= f(z,t)\bxi_z\,,\quad p=0~,
\end{equation}
where $f(z,t)=1+z(1 - 3\exp(-t))$, and $\bxi_z$ is the tangential vector field corresponding to the rigid rotation of the sphere about $z$ axis normalized to have $\|\bxi_z\|=2$. We note that $\bxi_z\in E$ and $\div_\Gamma \bu=0$.
The right hand side $\blf$ is defined such that \eqref{exact1} is exact solution to
\eqref{strongform-1}--\eqref{strongform-2}.

  Following the result in Corollary~\ref{cor1} we set $\tau=h^{-2}$ and vary the mesh size and time step.
  Results are shown in Figure~\ref{convergence} for $\nu=1$  and Figure~\ref{convergence1} for $\nu=0.01$.
  For both values of the viscosity parameter, the convergence plots follow the same pattern.  We observe the first order convergence in $L^2(0,T;H^1(\Gamma))$ velocity norm and the second order convergence to zero of the normal velocity component. Both  trends are in agreement with  \eqref{error1}. In $L^\infty(0,T;L^2(\Gamma))$ velocity norm we see second order convergence, which is better than was predicted by our analysis. The pressure converges with a rate between $O(h)$ and $O(h^2)$.

\subsection{Penalty parameter dependence}
\begin{figure}
\begin{center}
	\begin{tikzpicture}[scale=0.8]
	\def\a{2}
	\def\b{4}
	\def\c{-10}
	\def\T{1}
	
	\begin{axis}[xlabel={Time}, ylabel={Kinetic energy}, ymin=1.5, ymax=3.75, xmin=0, xmax=5, legend style={ cells={anchor=west}, legend pos=outer north east} ]
	
	\addplot[blue,mark=x,solid, mark options={solid},mark size=3pt,mark repeat={20},line width=1.pt]	table[x=Time, y=Kinetic ]
	{./eta=1.000000_l=5_nu=1_Plot_test20_convective_BDF2=0.010000.txt};
	
	\addplot[red,mark=square,solid,mark options={solid},mark size=2pt,mark repeat={20},line width=1.pt]	table[x=Time, y=Kinetic ]
	{./eta=2.000000_l=5_nu=1_Plot_test20_convective_BDF2=0.010000.txt};
	
	\addplot[black!30!green,mark=diamond,solid, mark options={solid},mark size=2pt,mark repeat={40},line width=1.pt]	table[x=Time, y=Kinetic ]
	{./eta=8.000000_l=5_nu=1_Plot_test20_convective_BDF2=0.010000.txt};
	
	\addplot[black!30!yellow,mark=o,loosely dotted,mark options={solid},mark size=3pt,mark repeat={40},line width=1.pt]	table[x=Time, y=Kinetic ]
	{./eta=32.000000_l=5_nu=1_Plot_test20_convective_BDF2=0.010000.txt};
	
	%
	
	\addplot[black,densely dashed,line width=2pt, samples=1000,domain=0:10]
	{
		0.2e1 / 0.35e2 * (0.7e1 * \a * \a * exp( (2 * \x / \T)) + \b * \b * exp( (2 * \x / \T)) + 0.2e1 * exp( (\x / \T)) * \b * \c + \c * \c) * exp(- (2 * \x / \T))
	};
	
	\legend{$ \tau=1$,
		$	\tau=2$,
		$ \tau=8$,
		$	\tau=32$,
		reference}
	\end{axis}
	\end{tikzpicture}
	\caption{Evolution of the kinetic energy for numerical solutions for different values of the penalty parameter $\tau$. Other parameters are fixed: $l=5$, $\Delta{}t=0.01$ and $\nu=1$.} \label{fig:tau}
\end{center}
\end{figure}
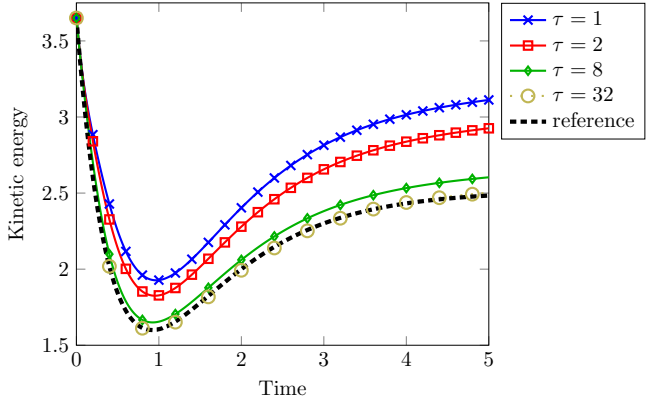

In this section we demonstrate that taking penalty parameter $\tau$ large enough is important for the numerical accuracy, but letting $\tau\to\infty$ leads to  larger errors in agreement with \eqref{error}. We consider two velocity--pressure pairs:
\begin{equation}\label{sol2}
\text{(a)}~~\bu= x f_1(z,t)\bxi_z\,,\quad  p=0,\qquad
\text{(b)}~~\bu= f_2(t)\bP \be_x\,,\quad  p=xy^3+z,
\end{equation}
where $f_1(z,t)=2+z(4 -10\exp(-t))$, $f_2(t)=1-\exp{}(1-6t)$, and calculate right hand sides such that (a) and (b) are exact solutions to \eqref{strongform-1}--\eqref{strongform-2}.  We choose both solutions such that $\Div_\Gamma\bu\neq0$, since otherwise the consistency term $E_s(\bu):\bH v_N$ vanishes for the spherical $\Gamma$ and does not contribute to the error equation. For both solutions we observe convergence of the method (plots not included) if we use the same refinement strategy as in section~\ref{s_num1}.
We next compute finite element solutions approximating \eqref{sol2}(a) with several values of the penalty parameter $\tau=2^k$, $k=0,\dots,5$. We fix mesh refinement level  $\ell=5$ and time step $\Delta t=0.01$.
Figure~\ref{fig:tau} shows the evolution of the kinetic energy for numerical solutions versus reference data. For $\tau=2^5$ the computed values match well with the reference curve.

%
%
%

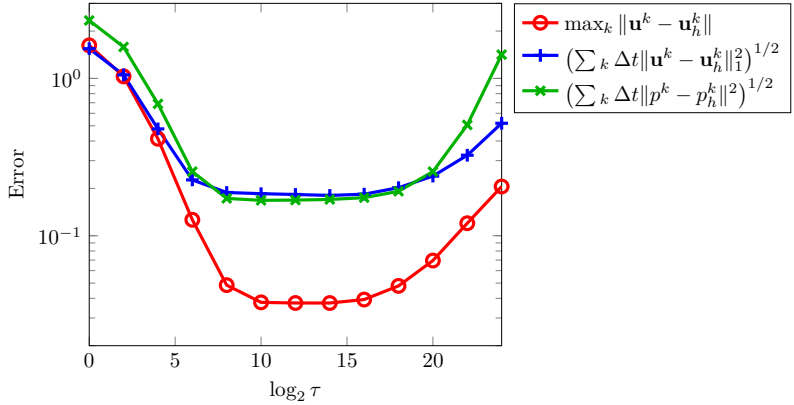
\begin{figure}
	\begin{center}
		\begin{tikzpicture}[scale=0.8]
		\def\vara{3}
		\def\varb{1}
		\begin{semilogyaxis}[ xlabel={$\log_2 \tau$}, ylabel={Error}, ymin=2E-2, xmax=24,xmin=0,
		, ymax=3,
		xtick={0,5,10,15,20},
		xticklabels={0,5,10,15,20},
		legend style={ cells={anchor=west}, legend pos=outer north east} ]
		\addplot+[red,mark=o,mark size=3pt,line width=1.5pt] table[x=level, y=t_max_uT] {./pmr_test95_eta_a=-1_l=4.dat};
		\addplot+[blue,mark=+,mark size=3pt,line width=1.5pt] table[x=level, y=t_H1] {./pmr_test95_eta_a=-1_l=4.dat};
		\addplot+[black!30!green,mark=x,mark size=3pt,line width=1.5pt] table[x=level, y=t_L2_p] {./pmr_test95_eta_a=-1_l=4.dat};
		
		
		\legend{
			$\max_k\|\bu^k- \bu_h^k\| $ ,
			$\left(\sum{}_k\, \Delta{}t\|\bu^k -\bu_h^k\|^2_{1}\right)^{1/2}$ ,
			$\left(\sum{}_k\,\Delta{}t\|p^k - p_h^k\|^2\right)^{1/2} $ ,
		}
		\end{semilogyaxis}
		\end{tikzpicture}
		\caption{Error norms for  velocity and pressure plotted against the penalty parameter $\tau$. The results were computed obtained with $\nu=1$, $\ell=4$, $\Delta{}t=0.02$, $t=[0,1]$.
		}
	\label{fig:tau_l=4}
	\end{center}
\end{figure}

Further we study how the error depends on the variation of $\tau$ and compute finite element solutions approximating \eqref{sol2}(b) with values of the penalty parameter $\tau=2^k$, $k=0,1,\dots,20$. Again the mesh refinement level $\ell$ and time step  are fixed. Results for $\ell=4$, $\Delta t=0.02$   are shown in Fig.~\ref{fig:tau_l=4} and they are in a good agreement with error estimate~\eqref{error}. We note a large plateau of optimal values for the penalty parameter, which makes its easier to choose a suitable $\tau$.

\subsection{Energy conservation for infinitesimal rigid transformation}

Tangential flows of infinitesimal rigid transformations on manifolds do not dissipate energy; see the first equality in \eqref{balance_separate}. We mentioned already that this property does not necessarily carry over to the discrete flow systems.
Numerical diffusion produced by such flows is due to the geometry and functional spaces approximations. In this section, we demonstrate this numerical phenomena and show that both viscous and inertia terms contribute to the numerical dissipation.
Moreover, the amount of the dissipation depends on the form of nonlinear terms. 
%
%
In all numerical experiments so far, we used the convective form. 
In the existing literature on numerical simulations of the surface Navier--Stokes equations, convective~\cite{reuther2015interplay,fries2017higher} and rotational~\cite{reuther2018solving,nitschke2017discrete} forms have been used. In the rotation form, one computes for the Bernoulli pressure instead of kinematic pressure and the nonlinear terms take the form $(\mbox{rot}_\Gamma\bu)\bn\times\bu=(\gradG{}\bu -\gradG^T{}\bu)\bu$.
 While equivalent for smooth solutions, this forms lead to discrete systems with possibly different numerical properties.

\begin{figure}
	\begin{center}
	\begin{tikzpicture}[scale=0.8]
	
	\begin{axis}[xlabel={Time}, ylabel={Kinetic energy}, ymin=0.96, ymax=1.01, xmin=0, xmax=10, legend style={ cells={anchor=west}, legend pos=outer north east} ]
	
	\addplot[black!30!yellow, dashed, mark=star,mark size=3pt,
	mark repeat={7}, mark options={solid},line width=1pt]
	table[x=Time, y=Kinetic]
	{./T=1_l=3_nu=0_Graph_test16_no_conv_BDF2=0.100000.txt};
	\addplot[black!30!yellow,solid,mark=star,mark size=3pt,
	mark repeat={7},  mark options={solid},line width=1.5pt]
	table[x=Time, y=Kinetic]
	{./T=1_l=4_nu=0_Graph_test16_no_conv_BDF2=0.100000.txt};
	
	\addplot[blue,mark=x,mark size=3pt,
	mark repeat={10}, dashed, mark options={solid},line width=1pt]
	table[x=Time, y=Kinetic]
	{./T=1_l=3_nu=0_Graph_test16_convective_BDF2=0.100000.txt};
	
	\addplot[blue,mark=x,mark size=3pt,
	mark repeat={10},mark options={solid},line width=1.5pt]
	table[x=Time, y=Kinetic]
	{./T=1_l=4_nu=0_Graph_test16_convective_BDF2=0.100000.txt};
	
	\addplot[red,dashed,mark=+,mark size=3pt,
	mark repeat={10}, mark options={solid},line width=1pt]
	table[x=Time, y=Kinetic]
	{./T=1_l=3_nu=0_Graph_test16_rotational_BDF2=0.100000.txt};
	\addplot[red,mark=+,mark size=2pt,
	mark repeat={10},line width=1.5pt]
	table[x=Time, y=Kinetic]
	{./T=1_l=4_nu=0_Graph_test16_rotational_BDF2=0.100000.txt};

%
%

	\addplot[black,dashed,line width=2pt,samples=10,domain=0:10] {1};
	
	\legend{
		$l=3$  w/o nonlinear,
		$l=4$  w/o nonlinear,
		$l=3$  convective ,
		$l=4$  convective,
		$l=3$  rotational,
		$l=4$  rotational,
		reference}
	
	\end{axis}
	\end{tikzpicture}
	\caption{ Kinetic energy evolution for FE Killing vector field, $\nu=0$.}
	\label{nu=0_stability}
\end{center}
\end{figure}

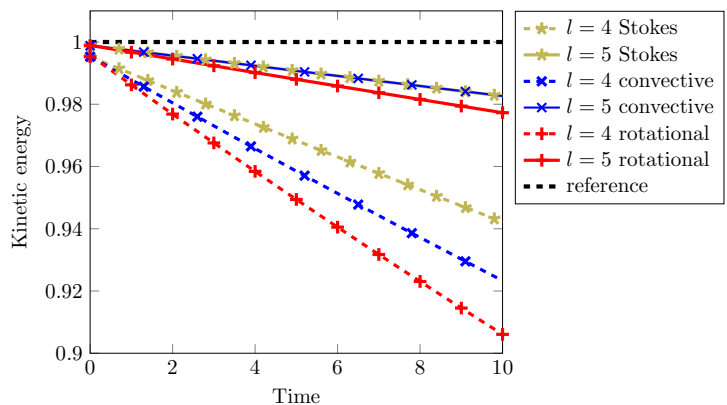
\begin{figure}
\begin{center}
	
	\begin{tikzpicture}[scale=0.8]
	
	\begin{axis}[xlabel={Time}, ylabel={Kinetic energy}, ymin=0.9, ymax=1.01, xmin=0, xmax=10, legend style={ cells={anchor=west}, legend pos=outer north east} ]
	
	\addplot[black!30!yellow,dashed,mark=star,mark size=3pt,
	mark repeat={7},mark options={solid},line width=1.5pt]
	table[x=Time, y=Kinetic]
	{./stability_l=4_mu=1_Graph_test16_no.txt};
	
		\addplot[black!30!yellow,solid,mark=star,mark size=3pt,
	mark repeat={7},  mark options={solid},line width=1.5pt]
	table[x=Time, y=Kinetic]
	{./eta=368.640747_l=5_nu=1_Plot_test16_none_BDF2=0.100000.txt};

	\addplot[blue,dashed,mark=x,mark size=3pt,mark options={solid},
	mark repeat={13},line width=1.5pt]
	table[x=Time, y=Kinetic]
	{./stability_l=4_mu=1_Graph_test16_full.txt};
	
	\addplot[blue,solid,mark=x,mark size=3pt,mark options={solid},
	mark repeat={13},line width=1pt]
	table[x=Time, y=Kinetic]
	{./eta=368.640747_l=5_nu=1_Plot_test16_convective_BDF2=0.100000.txt};

	\addplot[red,dashed,mark=+,mark size=3pt,mark options={solid},
	mark repeat={10},line width=1.5pt]
	table[x=Time, y=Kinetic]
	{./stability_l=4_mu=1_Graph_test16_rotation.txt};
	
	\addplot[red,solid,mark=+,mark size=3pt,mark options={solid},
	mark repeat={10},line width=1.5pt]
	table[x=Time, y=Kinetic]
	{./eta=368.640747_l=5_nu=1_Plot_test16_rotational_BDF2=0.100000.txt};
	
	\addplot[black,dashed,line width=2pt,samples=100,domain=0:10] {1};
	
	\legend{
		$l=4$  Stokes,
		$l=5$  Stokes,
		$l=4$  convective,
		$l=5$  convective,
		$l=4$  rotational,
		$l=5$  rotational,
		reference}
	
	\end{axis}
	\end{tikzpicture}
	\caption{Kinetic energy evolution for FE Killing vector field, $\nu=1$, $\Delta{}t=0.1$.}
	\label{nu=1_stability}
\end{center}
\end{figure}

Figures~\ref{nu=0_stability} and~\ref{nu=1_stability} show the kinetic energy history for the numerically simulated evolution of the
Killing vector field on the unit sphere. We set initial velocity equal to $P_1$ Lagrangian interpolant of $\bxi_z$ and run simulations for two refinement levels ($\ell=3$ and $\ell=4$) and two viscosity parameters ($\nu=0$ and $\nu=1$).
Experiment with $\nu=0$ illustrate the contribution of non-linear terms and corresponding pressure (through stabilization) to numerical energy dissipation, while experiment with  $\nu=1$ illustrate the contribution of numerical viscous stresses.
As should be expected, grid refinement lead to a rapid decrease of numerical diffusion in both cases. The results of computations without nonlinear term in Figures~\ref{nu=0_stability} (labeled by ``w/o nonlinear'') show that the volumetric stabilization and normal penalty alone do not produce any significant diffusion. The rotation form leads to more dissipative solution. The likely explanation is that the Bernoulli pressure contributes significantly to the discrete energy balance \eqref{balance_h} through the third term in the middle line.

\bibliographystyle{siam}
\bibliography{literatur}{}
\end{document}